\documentclass[12pt, reqno]{amsart}
\usepackage{amsmath,amssymb,amsthm}
\usepackage{mathrsfs}
\usepackage{amsfonts}
\usepackage{hyperref}
\usepackage{enumitem}
\usepackage{graphicx}
\usepackage{mathrsfs}
\usepackage{upgreek}
\usepackage{mathtools}
\usepackage{microtype}
\usepackage{tikz}       
\usepackage{amsmath}    
\theoremstyle{plain}
\newtheorem{theorem}{Theorem}[section]

\newtheorem{lemma}{Lemma}[section]

\theoremstyle{definition}

\theoremstyle{remark}
\newtheorem{remark}{Remark}[section]

\newcommand{\R}{\mathbb{R}}

\begin{document}
	
\title{Hausdorff dimension of the graph of a random vector-valued Weierstrass function}

\author{Jun Jason Luo}  \address{College of Mathematics and Statistics,  Chongqing University, Chongqing, 401331, P.R. China} \email{jun.luo@cqu.edu.cn}

\author{Zi-Rui Zhang}  \address{College of Mathematics and Statistics,  Chongqing University, Chongqing, 401331, P.R. China} \email{202306021017@\allowbreak stu.cqu.edu.cn}

\begin{abstract}
	Let $\{\theta_n,\lambda_n:n\geq0\}$ be mutually independent
	random variables, each uniformly distributed on $[0,1]$,
	and write $\Theta=(\theta_n)_{n\geq0}$ and
	$\Lambda=(\lambda_n)_{n\geq0}$. The  random vector-valued Weierstrass function is given by
		$$
		f_{\Theta,\Lambda}(x)=
		\left(
		\sum_{n=0}^{\infty} a^n\cos\bigl(2\pi (b^n x+\theta_n)\bigr),\ 
		\sum_{n=0}^{\infty} a^n\sin\bigl(2\pi (b^n x+\lambda_n)\bigr)
		\right), \;  x\in[0,1],
		$$ where $0<a<1<b,\ ab> 1$. The Hausdorff dimension of the graph of this function is proved to be
	$$\dim_H G(f_{\Theta,\Lambda}) = \min\left\{-\frac{\log b}{\log a}, \, 3 +2\frac{\log a}{\log b}\right\} \quad \text{a.s.}$$
\end{abstract}

	
	\keywords{Hausdorff dimension; Box dimension; Weierstrass function}
	
	\subjclass[2020]{Primary 28A78; Secondary 37C45}
	
	\thanks{The research was supported by the Natural Science Foundation of Chongqing (No.CSTB2023NSCQ-MSX0553)}

	\maketitle
	
	\section{Introduction}
	
In 1872, Weierstrass \cite{Weierstrass1872} constructed the pathological function in series form
\begin{equation}\label{w-cos function}
	W(x)=\sum_{n=0}^\infty a^n\cos(2\pi b^n x),\quad x\in [0,1],
\end{equation}
where $0<a<1<b,\ ab\ge 1$, and gave the first rigorous proof that $W$ is continuous everywhere but nowhere differentiable for some $a$ and $b$. This provided a basic example of the distinction between continuity and differentiability. Hardy \cite{Hardy1916} refined the proof of non-differentiability and laid a solid theoretical foundation. In 1977, Mandelbrot \cite{Mandelbrot1977} proposed the famous conjecture:
\[
\dim_H(G(W))=\dim_B(G(W))=2+\frac{\log a}{\log b}
\] where $\dim_H, \dim_B$ are Hausdorff and box dimensions and $G(W)$ denotes the graph of $W$ over $[0,1]$. This dimension problem  has attracted considerable attention in fractal geometry.

For integer $b>1$, Kaplan, Mallet-Paret, and Yorke \cite{Kaplan1984} used a dynamical-systems approach to establish the box dimension formula.  Mauldin and Williams \cite{Mauldin1986} proved that $\dim_H(G(W))$ is bounded below by $2+{\log a}/{\log b}-O(1/{\log b})$ for sufficiently large $b$. Later, Przytycki and Urbański \cite{Przytycki1989} showed that $\dim_H(G(W))>1$. Ledrappier \cite{Ledrappier1992} used Pesin theory to obtain almost-sure dimension results for the sawtooth-type Weierstrass function.  Solomyak \cite{Solomyak1995} provided crucial support for the conjecture using transversality arguments and absolute continuity of Bernoulli convolutions. 

In related work, Hunt \cite{Hunt1998} constructed the randomized Weierstrass function by introducing an independent and identically distributed (i.i.d.) random phase sequence \(\Theta=(\theta_n)_{n\geq0}\):
\[
W_{\Theta}(x) = \sum_{n=0}^{\infty} a^n \cos\left(2\pi \left(b^n x + \theta_n\right)\right).
\]
He proved that with probability one, the Hausdorff dimension of the graph of $W_{\Theta}$ equals the conjectured value, i.e.,
\[
\dim_H(G(W_{\Theta}))=2+\frac{\log a}{\log b} \qquad \text{a.s.}
\]
An analogous result was established in \cite{Rom2014} by Romanowska by randomizing the amplitudes independently at each level. Heurteaux \cite{Heurteaux2003} studied the   following  general Weierstrass-type function with random phases
\[
f_{a,b}^{\phi,\Theta}(x) = \sum_{n=0}^\infty a^n \phi(b^n x + \theta_n)
\]
where $\phi$  is a non-constant $1$-periodic Lipschitz function. This function is $\beta$-H\"older continuous for every
phase sequence, where $\beta=-\log a/\log b$. Using lower bounds for its local oscillations, he proved that its graph  has box and packing dimension $2-\beta$ almost surely.   Baránski \cite{Baranski2012} extended the result to the   deterministic Weierstrass-type function $f^\phi(x) = \sum_{n=0}^\infty a_n \phi(b_n x +\theta_n)$ and determined  both the Hausdorff and box dimensions of its graph, with the dimension depending on the rapidly growing frequencies satisfying \(a_{n+1}/a_n\to 0\) and \(b_{n+1}/b_n\to \infty\).

A substantial advance toward Mandelbrot’s conjecture was achieved in \cite{Baranski2014}, where Baránski, Bárány and Romanowska verified the conjecture for all integers \(b>1\) over a large parameter class of $a$. Soon later, Shen \cite{Shen2018} fully settled the conjecture via tools from ergodic theory and dynamical systems, assuming integer \(b>1\) and \(1/b<a<1\). Ren and Shen \cite{RenShen2021} further extended the preceding result. For every integer \(b>1\) with \(1/b<a<1\) and any real-analytic periodic function \(\phi:\mathbb{R}\to\mathbb{R}\), by defining 
\begin{equation}\label{w-function 1}
	f_{a,b}^\phi(x) = \sum_{n=0}^\infty a^n \phi(b^n x), 
\end{equation} they proved that the function \(f_{a,b}^\phi\) falls into exactly two categories: it is either real-analytic, whose graph has Hausdorff dimension $1$, or fractal, with its graph possessing Hausdorff dimension \(2+{\log a}/{\log b}\).

Meanwhile, partial progress has been achieved concerning the dimensional properties of graphs of vector-valued Weierstrass functions. It is known that when $b>1$ is an integer, the classical Weierstrass function \eqref{w-cos function} is the real part of the lacunary complex power series
\[
w(z) = \sum_{n=0}^{\infty} a^n z^{b^n}, \quad z\in \mathbb{C},\ |z| \leq 1
\]
on the unit circle. Salem and Zygmund \cite{Salem1945} and Kahane, Weiss, and Weiss \cite{Kahane1963} showed that, for each integer
$b\geq2$ and $a$ sufficiently close to $1$, the image of the unit circle under $w$ contains a nonempty open subset of $\mathbb C$.  Under the same assumptions, Baránski \cite{Baranski2002} further proved that the box dimension of the graph of the vector-valued Weierstrass function
\[
f_{a, b}(x)= \left( \sum_{n=0}^{\infty} a^n \cos(2\pi b^n x), \,  \sum_{n=0}^{\infty} a^n \sin(2\pi b^n x) \right)
\]
satisfies
\[
\dim_B(G(f_{a,b})) = 3 + 2\frac{\log a }{\log b},
\]
where the graph $G(f_{a,b})$ is a subset of $\mathbb{R}^3$ (see Figure \ref{fig1}). Ren \cite{Ren2023} extended Baránski’s result by replacing trigonometric functions with a vector-valued Lipschitz function $\phi:\mathbb{R}\to\mathbb{R}^2$ of period $1$. He proved that if $\mathbb{R}^2\setminus\phi(\mathbb{R})$ is disconnected, then for any integer $b>1$ and for $a$ sufficiently close to $1$, the box dimension of the graph of the vector-valued Weierstrass-type function of the form \eqref{w-function 1} is also given by $3 + 2{\log a }/{\log b}.$
Recently, Ren and Shen \cite{RenShen2024} improved their one-dimensional result into high-dimensional vector-valued Weierstrass functions. A comprehensive account of dimension results for graphs of Weierstrass-type functions can be found in Baránski's excellent survey paper \cite{Baranski2015}.	
\begin{figure}[h]
\centering
\includegraphics[width=6.5cm]{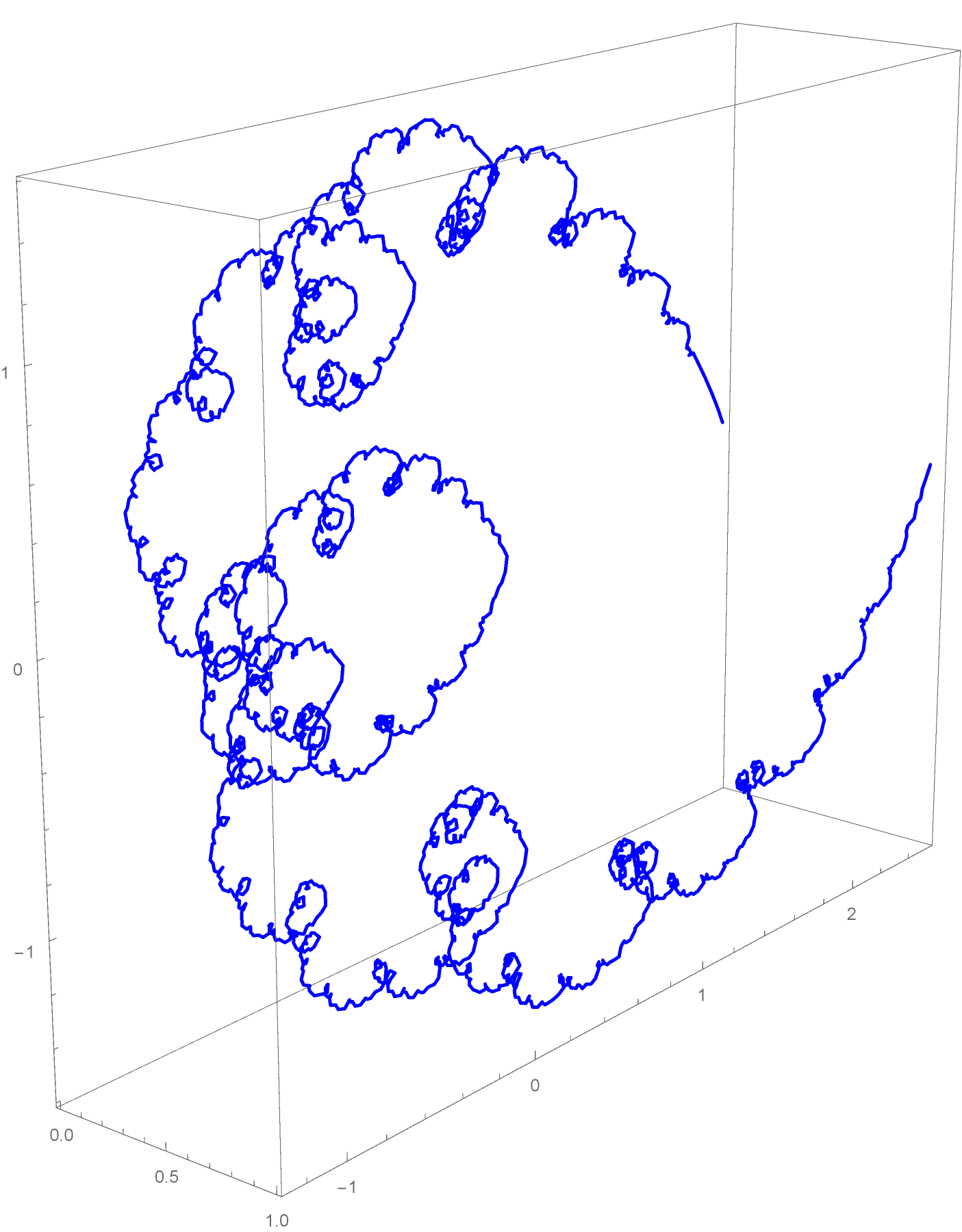}
\caption{The deterministic graph $G(f_{a,b})$ for $a=0.6$, $b=2$.}\label{fig1}
	\end{figure}
	
Motivated by these results, we study random vector-valued Weierstrass functions and determine the Hausdorff and box dimensions of their graphs.
Let $H=[0,1]^{\mathbb N_0}$ carry its product Borel sigma-algebra and product Lebesgue probability measure $\mathbb Q$. Equip $H\times H$ with the probability measure $\mathbb P=\mathbb Q\otimes\mathbb Q$, and let
\[
\Theta=(\theta_n)_{n\geq0},
\qquad
\Lambda=(\lambda_n)_{n\geq0}
\]
be the coordinate sequences. Thus the random variables $\theta_n,\lambda_n$, $n\geq0$, are mutually independent and uniformly distributed on $[0,1]$. All almost-sure statements below refer to $\mathbb P$.

For $0<a<1<b$ with $ab>1$, define the random vector-valued
Weierstrass function $f_{\Theta,\Lambda}:[0,1]\to\mathbb R^2$ by
\begin{equation}\label{random-vector-valued function}
	f_{\Theta,\Lambda}(x)
	=\left(
	\sum_{n=0}^{\infty}a^n
	\cos\bigl(2\pi(b^nx+\theta_n)\bigr),\,
	\sum_{n=0}^{\infty}a^n
	\sin\bigl(2\pi(b^nx+\lambda_n)\bigr)
	\right).
\end{equation}
Since $\sum_{n=0}^{\infty}a^n<\infty$, both series
converge uniformly in $x$ and the phase sequences.
Consequently, for every $(\Theta,\Lambda)\in H\times H$,
the function $f_{\Theta,\Lambda}$ is continuous on $[0,1]$,
and its graph is a compact subset of $\mathbb R^3$.
The partial sums are jointly measurable in $x$ and
$(\Theta,\Lambda)$, so their limit is jointly measurable
as well.

For the above random fractal function, the main dimension result is stated as follows:

\begin{theorem}\label{main thm}
	Let \(f_{\Theta,\Lambda}: [0,1] \to \mathbb{R}^2\) be the random vector-valued Weierstrass function defined in \eqref{random-vector-valued function}. The Hausdorff dimension of the graph satisfies
	\[
	\dim_H G(f_{\Theta,\Lambda}) =\dim_B G(f_{\Theta,\Lambda}) = \min\left\{-\frac{\log b}{\log a}, \, 3 +2\frac{\log a}{\log b}\right\} \quad \text{a.s}.
	\]
\end{theorem}

H\"older continuity gives upper bounds for the Hausdorff
and upper box dimensions for every pair of phase sequences.
The main difficulty lies in establishing the matching
lower bounds. Set $\beta=-\log a/\log b$. When $\beta\in[1/2,1)$,
we first determine the Hausdorff dimension of the image
$f_{\Theta,\Lambda}([0,1])$ by Fourier estimates.
The matching lower bound for the graph then follows by
orthogonal projection onto the image.
When $\beta\in(0,1/2)$, our proof closely follows Hunt's
proof for the scalar random-phase Weierstrass function
\cite{Hunt1998}, with the increment-density estimates
adapted to two independent coordinates.

The potential-theoretic approach to Hausdorff dimension
goes back to Frostman's work in 1935.
For the energy criterion used here, see
\cite[Chapter 4]{Falconer2014} and \cite{Mattila2015}.  For the definitions of Hausdorff and box dimensions and the dimension estimates for H\"older maps, see \cite[Chapters 2--3]{Falconer2014}. We write $\dim_B$ when the lower and upper box dimensions coincide.  

\begin{remark}
	A related random complex Weierstrass function is
	\[
	g_\Theta(x)
	=\sum_{n=0}^\infty a^n e^{2\pi i(b^nx+\theta_n)},
	\]
	whose real and imaginary parts share the same phase sequence. Recently, 
	Lai, Lau and Zhang \cite{LaiLauZhang2026} proved that for  $\beta=-\log a/\log b$,
	\[
	\dim_H G(g_\Theta)=\min\{1/\beta,3-2\beta\}  \quad \text{a.s}.
	\]

	In contrast, the two coordinates of $f_{\Theta,\Lambda}$
	have independent phase sequences. We use this independence
	to express the joint characteristic function and density
	of each increment as products of the corresponding
	one-dimensional factors. For $g_\Theta$, each nonzero
	single-level increment is uniformly distributed on a circle,
	and its real and imaginary parts are dependent.
	Thus these factorizations do not apply to $g_\Theta$,
	even though the two functions have the same almost-sure graph dimensions.
\end{remark}

\section{Proof of Theorem \ref{main thm}}
Set $\beta = - {\log a }/{\log b}.$ Then  $\beta \in (0,1)$ and $a = b^{-\beta}$. The random vector-valued Weierstrass  function in \eqref{random-vector-valued function} can be rewritten as: for $x\in [0,1]$,
	\begin{equation*}
		f_{\Theta,\Lambda}(x) = \left( \sum_{n=0}^{\infty} b^{-\beta n} \cos 2\pi (b^n x + \theta_n),\,  \sum_{n=0}^{\infty} b^{-\beta n} \sin 2\pi (b^n x + \lambda_n) \right).
	\end{equation*}
More precisely, let
\begin{equation}\label{eq-phi-n}
	\varphi_n:=\varphi_n(x,y)=2b^{-\beta n}\sin(\pi b^n(x-y)). 
\end{equation} The sum-to-product formulas in trigonometry imply that
\begin{equation}\label{eq- vector-valued funct.}
	f_{\Theta,\Lambda}(x) -f_{\Theta,\Lambda}(y)=(a_1(x,y), a_2(x,y))
\end{equation}
 where 	
	\begin{eqnarray*}
	a_1(x,y) &=& \sum_{n=0}^{\infty} -\varphi_n(x,y) \sin\left(\pi\left(b^n x + b^n y +2\theta_n\right)\right), \\
	a_2(x,y) &=& \sum_{n=0}^{\infty} \varphi_n(x,y) \cos\left(\pi\left(b^n x + b^n y +2\lambda_n\right)\right).
   \end{eqnarray*}
	
Define the graph map of $f_{\Theta,\Lambda}$ by $$G(f_{\Theta,\Lambda}, x)= (x,  f_{\Theta,\Lambda}(x)): [0,1]\to \R^3.$$  Then the graph of  $f_{\Theta,\Lambda}$ over $[0,1]$ is given by $$G(f_{\Theta,\Lambda})=G(f_{\Theta,\Lambda}, [0,1]).$$ 

We include proofs of the following two elementary estimates for completeness.
	
	\begin{lemma}\label{lem-Holder}
	Both $f_{\Theta,\Lambda}(x)$ and $G(f_{\Theta,\Lambda},x)$ are $\beta$-H\"older continuous on $[0,1]$ for every pair of phase sequences $(\Theta,\Lambda)$.
	\end{lemma}
	
	\begin{proof}
		For any \( x, y \in [0,1], x\ne  y\),
		\begin{eqnarray*}
				|f_{\Theta,\Lambda}(x) -f_{\Theta,\Lambda}(y)| &=& \sqrt{(a_1(x,y))^2+(a_2(x,y))^2} \\				 
				&\leq& 2\sum_{n=0}^{\infty}|\varphi_n(x,y)|\\
&\leq& 4 \sum_{n=0}^{\infty}  b^{-\beta n} \min\left\{1, \pi b^n |x - y|\right\} \\
				&\leq& 4  \pi \sum_{n=0}^{m-1} b^{(1-\beta) n} |x - y| + 4 \sum_{n=m}^{\infty} b^{-\beta n} \\
				&=& 4\pi \frac{b^{(1-\beta) m} - 1}{b^{1-\beta} - 1} |x - y| + 4 \frac{b^{-\beta m}}{1 - b^{-\beta}}
		\end{eqnarray*}
		for all \( m \ge 1 \). The case $|x-y|=1$ follows directly from boundedness, with the same displayed constant.  For $|x-y|<1$, let \( m \) be the positive integer satisfying \( b^{-m} \leq |x - y| < b^{-m+1} \). It follows that
		\begin{eqnarray}\label{equ- constant}
			|f_{\Theta,\Lambda}(x) -f_{\Theta,\Lambda}(y)| &\leq& 4\pi \frac{(b / |x - y|)^{1-\beta}}{b^{1-\beta} - 1} |x - y| + 4\frac{|x - y|^{\beta}}{1 - b^{-\beta}} \nonumber \\
				&\leq& \left( \frac{4\pi b^{1-\beta}}{b^{1-\beta} - 1} + \frac{4}{1 - b^{-\beta}} \right) |x - y|^{\beta} \nonumber \\
				&:=& C |x - y|^{\beta}. 
		\end{eqnarray}		
		Moreover, $$|G(f_{\Theta,\Lambda},x)-G(f_{\Theta,\Lambda},y)|\le|x-y|+|f_{\Theta,\Lambda}(x) -f_{\Theta,\Lambda}(y)|\le (C+1)|x-y|^\beta.$$		
	\end{proof}

	\begin{lemma}\label{lem-lowerbound}
		For every pair of phase sequences $(\Theta,\Lambda)$,
\[
\dim_H G(f_{\Theta,\Lambda})\leq\overline{\dim}_B G(f_{\Theta,\Lambda})
\leq\min\{1/\beta,3-2\beta\}.
\]
	\end{lemma}
	
	\begin{proof}
The $\beta$-H\"older continuity in Lemma~\ref{lem-Holder}
gives $\dim_H G(f_{\Theta,\Lambda})\leq1/\beta$.
For $0<\delta<1$, divide $[0,1]$ into at most
$2\delta^{-1}$ intervals of length at most $\delta$.
The graph over each interval has diameter at most
$(C+1)\delta^\beta$ and can therefore be covered by
one ball of that radius. Hence
\[
N\bigl(G(f_{\Theta,\Lambda}),(C+1)\delta^\beta\bigr)
\leq2\delta^{-1},
\]
where $N(E,r)$ denotes the least number of balls of
radius $r$ needed to cover $E$. Consequently,
\[
\overline{\dim}_B G(f_{\Theta,\Lambda})
\leq\limsup_{\delta\downarrow0}
\frac{\log(2\delta^{-1})}
{-\log\bigl((C+1)\delta^\beta\bigr)}
=\frac1\beta.
\]
		
	On the other hand,	take an integer \( m>1 \) and divide \( [0,1] \) into \( m \) subintervals of equal length. Let $J$ be one of the subintervals.  H\"older continuity implies that
		\[
		\sup_{x,y\in J}|f_{\Theta,\Lambda}(x)-f_{\Theta,\Lambda}(y)| \le C m^{-\beta}
		\] where $C$ is the coefficient in \eqref{equ- constant}. 	The graph of the function on $J$ can be covered by a cuboid of size
		\( m^{-1}\times C m^{-\beta}\times C m^{-\beta} \), and hence by at most $(C m^{1-\beta} + 1)^2$ cubes of side length \( 1/m \). Thus the whole graph  $G(f_{\Theta,\Lambda})$ over $[0,1]$ can be covered by at most $m\bigl(C m^{1-\beta}+1\bigr)^2=m^{3-2\beta}\bigl(C+m^{\beta-1}\bigr)^2$
		cubes of side length \( 1/m \). Therefore,
		\[
		\dim_H G(f_{\Theta,\Lambda})
		\le \overline{\dim}_B G(f_{\Theta,\Lambda})
		\le \limsup_{m\to\infty}
		\frac{\log\bigl(m^{3-2\beta}(C+m^{\beta-1})^2\bigr)}{-\log(1/m)}
		=3-2\beta
		\] where $\overline{\dim}_B$ denotes the upper box dimension, proving the lemma.
	\end{proof}

\begin{theorem}\label{thm2.1}
 	$\dim_H f_{\Theta,\Lambda}([0,1]) = \min\{2, {1}/{\beta}\}$ a.s.
 \end{theorem}
 	
 \begin{proof}
 	 For every pair of phase sequences $(\Theta,\Lambda)$, $\dim_H f_{\Theta,\Lambda}([0,1]) \le 2$, and by the H\"older continuity, we have
 	\[
 	\dim_H f_{\Theta,\Lambda}([0,1]) \le \frac{1}{\beta} \dim_H([0,1]) = \frac{1}{\beta}.
 	\]
 	That proves $\dim_H f_{\Theta,\Lambda}([0,1]) \le \min\{2, {1}/{\beta}\}$. 
 	
 	For the converse inequality, set $L=1/(2b^5)$ and let
 	$\mathcal L$ denote Lebesgue measure on $[0,1]$.
 	Define the push-forward measure
 	\[
 	\nu:=\mathcal L|_{[0,L]}\circ f_{\Theta,\Lambda}^{-1},
 	\]
 	which is supported on $f_{\Theta,\Lambda}([0,L])$,
 	and put $B=[0,L]^2$.
 	For $\xi=(\xi_1,\xi_2)\in\mathbb R^2$, the Fourier
 	transform of $\nu$ is
 	\[
 	\widehat\nu(\xi)
 	=\int_{\mathbb R^2}e^{-2\pi i\langle\xi,w\rangle}\,d\nu(w)
 	=\int_0^L e^{-2\pi i\langle\xi,f_{\Theta,\Lambda}(x)\rangle}\,dx.
 	\]
 	Then
 \begin{eqnarray*}
 	 	|\hat{\nu}(\xi)|^2
 	 &=& \int_{B} e^{-2\pi i \langle \xi, f_{\Theta,\Lambda}(x) - f_{\Theta,\Lambda}(y) \rangle} \, dx\,dy \\
 	 &=& \int_{B} e^{-2\pi i (\xi_1 a_1(x,y) + \xi_2 a_2(x,y))} \, dx\,dy \\ 	 
 	 &=& \int_{B} e^{-2\pi i \xi_1 a_1(x,y)} e^{-2\pi i \xi_2 a_2(x,y)} \, dx\,dy.
 \end{eqnarray*}

 	Taking expectation and using Fubini's theorem for the bounded integrand:
 \begin{eqnarray*}
 	\mathbb{E}\left(|\hat{\nu}(\xi)|^2\right)
 	&=& \int_{B} \mathbb{E}\left(e^{-2\pi i \xi_1 a_1(x,y)} e^{-2\pi i \xi_2 a_2(x,y)}\right)\, dx\,dy \\
 	&=& \int_{B} \mathbb{E}\left(e^{-2\pi i \xi_1 a_1(x,y)}\right) \mathbb{E}\left(e^{-2\pi i \xi_2 a_2(x,y)}\right) \, dx\,dy.
 \end{eqnarray*}
 	
By substituting the expressions for \(a_1(x,y), a_2(x,y)\), we obtain
 	\[
 	e^{-2\pi i \xi_1 a_1(x,y)}
  	= \prod_{n=0}^\infty e^{2\pi i \xi_1 \varphi_n(x,y) \sin(\pi(b^n x + b^n y + 2\theta_n))},
 	\]
 	\[
 	e^{-2\pi i \xi_2 a_2(x,y)}
 	= \prod_{n=0}^\infty e^{-2\pi i \xi_2 \varphi_n(x,y) \cos(\pi(b^n x + b^n y + 2\lambda_n))}.
 	\]
 	
 	For finite partial sums, independence gives the product
 	formula for their characteristic functions.
 	The partial sums converge to the corresponding increments,
 	and the complex exponentials have modulus one. The dominated convergence theorem therefore yields
 	\begin{align*}
 		\mathbb E\left(e^{-2\pi i\xi_1 a_1(x,y)}\right)
 		&=\lim_{N\to\infty}
 		\prod_{n=0}^{N}
 		\int_0^1
 		e^{2\pi i\xi_1\varphi_n(x,y)\sin(2\pi\theta)}
 		\,d\theta\\
 		&=\prod_{n=0}^{\infty}
 		J_0\bigl(2\pi\xi_1\varphi_n(x,y)\bigr)
 	\end{align*}
  	where $$J_0(u) := \int_0^1 e^{i u \sin(2\pi \theta)} d\theta = \int_0^1 e^{i u \cos(2\pi \theta)} d\theta$$ is the Bessel function of order zero satisfying $J_0(-u)=J_0(u)$ (see \cite {Mattila2015} for more details).
 	
 	Similarly,
 	\[
 	\mathbb{E}\left(e^{-2\pi i \xi_2 a_2(x,y)}\right)
 	= \prod_{n=0}^\infty J_0\left(2\pi \xi_2 \varphi_n(x,y)\right).
 	\]
 	Therefore,
 	\[
 	\mathbb{E}\left(|\hat{\nu}(\xi)|^2\right)
 	= \int_B \prod_{n=0}^\infty \left(J_0(2\pi \xi_1 \varphi_n(x,y)) J_0(2\pi \xi_2 \varphi_n(x,y))\right)\,dx\,dy.
 	\]

For $0<s<2$, the Fourier representation of $s$-energy
\cite{Mattila2015} gives
\begin{align*}
	I_s(\nu)=\iint_{\mathbb R^2\times\mathbb R^2}
	|u-v|^{-s}\,d\nu(u)\,d\nu(v)=\gamma(2,s)\int_{\mathbb R^2}
	|\xi|^{s-2}|\widehat\nu(\xi)|^2\,d\xi,
\end{align*}
where
\[
\gamma(2,s)
=\pi^{s-1}\frac{\Gamma(1-s/2)}{\Gamma(s/2)}.
\]
For $x\ne y$, define
\[
P_{x,y}(t)
=\prod_{n=0}^\infty J_0\bigl(2\pi t\varphi_n(x,y)\bigr)
\]
and
\[
\Psi_s(x,y)
=\int_{\mathbb R}|t|^{(s-2)/2}|P_{x,y}(t)|\,dt.
\]
We first estimate $\Psi_s(x,y)$.  Define
\[
B_k=B\cap\left\{(x,y):\frac1{2b^{k+1}}<|x-y|\leq\frac1{2b^k}\right\},
\qquad k\geq5.
\]
The sets $B_k$, $k\geq5$, partition $B$ apart from its
diagonal, which has zero two-dimensional Lebesgue measure.
For each fixed $x\in[0,L]$, the set of $y$ such that
$(x,y)\in B_k$ has length at most $b^{-k}$.
Hence, by Fubini's theorem,
\[
\mathcal L^2(B_k)
\leq \int_0^L b^{-k}\,dx
=Lb^{-k}.
\]

For $(x,y)\in B_k$ and $n=k-4,\ldots,k$,
\[
0<\pi b^n|x-y|\leq\pi/2.
\]
Using $|\sin u|\geq 2|u|/\pi$ for $|u|\leq\pi/2$
and the standard estimate
\[
|J_0(u)|\leq \min\{1,C|u|^{-1/2}\}, \quad u\in\mathbb R\setminus\{0\},
\]
where $C>0$ is an absolute constant (see \cite[Section 3.3]{Mattila2015}), we obtain, for $t\ne0$,
\begin{align*}
|J_0(2\pi t\varphi_n)|
&\leq C(4\pi|t|)^{-1/2}b^{\beta n/2}
       |\sin(\pi b^n(x-y))|^{-1/2}\\
&\leq C(4\pi|t|)^{-1/2}
       b^{[k+1+(\beta-1)n]/2}.
\end{align*}
Since $\sum_{n=k-4}^k n=5k-10$, multiplication gives
\[
\prod_{n=k-4}^k|J_0(2\pi t\varphi_n)|
\leq C^5(4\pi|t|)^{-5/2}
 b^{(5k\beta-10\beta+15)/2}.
\]
All remaining Bessel factors have modulus at most one.
For $(x,y)\in B_k$, write
\[
\Psi_s(x,y)=I_1+I_2,
\]
where
\[
I_1=\int_{|t|\leq b^{\beta k}}
|t|^{(s-2)/2}|P_{x,y}(t)|\,dt,
\qquad
I_2=\int_{|t|>b^{\beta k}}
|t|^{(s-2)/2}|P_{x,y}(t)|\,dt.
\]
Since $s>0$,
\[
I_1\leq2\int_0^{b^{\beta k}}t^{(s-2)/2}\,dt
=\frac4s b^{\beta sk/2}.
\]
Using the preceding product estimate and the evenness of
$|P_{x,y}(t)|$, we obtain
\begin{align*}
	I_2
	&\leq
	2C^5(4\pi)^{-5/2}b^{(5k\beta-10\beta+15)/2}
	\int_{b^{\beta k}}^\infty t^{(s-7)/2}\,dt\\
	&=\frac{4C^5(4\pi)^{-5/2}}{5-s}
	b^{(5k\beta-10\beta+15)/2}
	b^{\beta k(s-5)/2}\\
	&=\frac{4C^5(4\pi)^{-5/2}}{5-s}
	b^{(15-10\beta)/2}b^{\beta sk/2}.
\end{align*}
The integral converges because $s<2$.
Combining these bounds gives
\begin{equation}\label{eq-Phi}
	\Psi_s(x,y)\leq C_{s,b,\beta}b^{\beta sk/2},
	\qquad (x,y)\in B_k,
\end{equation}
where $C_{s,b,\beta}$ is independent of $k,x,y$.

The expectation identity above gives
\[
0\leq\mathbb E\left(|\widehat\nu(\xi)|^2\right)
\leq\int_B
|P_{x,y}(\xi_1)|\,|P_{x,y}(\xi_2)|\,dx\,dy.
\]
Since $|\xi|^2\geq|\xi_1||\xi_2|$ and $s-2<0$, we have
\[
|\xi|^{s-2}
\leq|\xi_1|^{(s-2)/2}|\xi_2|^{(s-2)/2}
\]
away from the coordinate axes, which have zero
two-dimensional Lebesgue measure. 
By the Fourier representation of $s$-energy and
Tonelli's theorem,
\begin{align*}
	\mathbb E\left(I_s(\nu)\right)
	&=\gamma(2,s)\int_{\mathbb R^2}
	|\xi|^{s-2}\,
	\mathbb E\left(|\widehat\nu(\xi)|^2\right)\,d\xi\\
	&\leq\gamma(2,s)\int_B\int_{\mathbb R^2}
	|\xi_1|^{(s-2)/2}|\xi_2|^{(s-2)/2}
	|P_{x,y}(\xi_1)|\,|P_{x,y}(\xi_2)|
	\,d\xi\,dx\,dy\\
	&=\gamma(2,s)\int_B
	\left(
	\int_{\mathbb R}|t|^{(s-2)/2}|P_{x,y}(t)|\,dt
	\right)^2dx\,dy\\
	&=\gamma(2,s)\int_B\Psi_s(x,y)^2\,dx\,dy\\
	&=\gamma(2,s)\sum_{k=5}^{\infty}
	\int_{B_k}\Psi_s(x,y)^2\,dx\,dy\\
	&\leq\gamma(2,s)C_{s,b,\beta}^{\,2}
	\sum_{k=5}^{\infty}
	b^{\beta sk}\mathcal L^2(B_k)\\
	&\leq\gamma(2,s)C_{s,b,\beta}^{\,2}L
	\sum_{k=5}^{\infty}b^{(\beta s-1)k}
	<\infty
\end{align*}
whenever $0<s<\min\{2,1/\beta\}$, since $\beta s<1$.  Thus $I_s(\nu)<\infty$
almost surely. Since $\nu$ is a nonzero finite measure supported on
$f_{\Theta,\Lambda}([0,L])$,  the energy criterion
\cite[Chapter 4]{Falconer2014} implies
\[
\dim_H f_{\Theta,\Lambda}([0,L])\geq s
\qquad\text{a.s}.
\]
Choose a sequence $0<s_j\uparrow\min\{2,1/\beta\}$. On the intersection of the corresponding probability-one events,
the preceding lower bound holds for every $s_j$. Letting
$j\to\infty$ and using $[0,L]\subset[0,1]$, we conclude that
\[
\dim_H f_{\Theta,\Lambda}([0,1])
\geq\dim_H f_{\Theta,\Lambda}([0,L])
\geq\min\{2,1/\beta\}
\qquad\text{a.s}.
\]
 \end{proof}

\begin{proof}[Proof of Theorem \ref{main thm}]
	Define the projection  $P: G(f_{\Theta,\Lambda}) \to f_{\Theta,\Lambda}([0,1])$ from the graph onto the image by 
	$$P(x, f_{\Theta,\Lambda}(x))= f_{\Theta,\Lambda}(x), \quad x\in [0,1].$$
	Trivially $P$ is a Lipschitz function.  Hence $\dim_H G(f_{\Theta,\Lambda})\ge \dim_H f_{\Theta,\Lambda}([0,1]).$
	
    If $\beta\in [\frac12, 1)$, 	it follows from  Lemma \ref{lem-lowerbound} and Theorem \ref{thm2.1} that
	$$ \frac {1}{\beta} \ge \dim_H G(f_{\Theta,\Lambda})\ge \dim_H f_{\Theta,\Lambda}([0,1])=\frac{1}{\beta} \quad \text{a.s.}$$
	That proves  $\dim_H G(f_{\Theta,\Lambda}) = \frac {1}{\beta}$ a.s.
	
	 If $\beta\in (0, \frac12)$,  we let   \( \mu:= \mu_{\Theta,\Lambda} \) be the induced measure on $\mathbb{R}^3$ through the Lebesgue measure  \( {\mathcal L} \) and $f_{\Theta,\Lambda}$ in the following way: for any Borel set \( V \subset \mathbb{R}^3 \), define
	\begin{equation*}
		\mu (V) = {\mathcal L}\bigl(\{t \in [0,1] : (t, f_{\Theta,\Lambda}(t)) \in V\}\bigr).
	\end{equation*} 
	The measure $\mu$ is supported on the graph  $G(f_{\Theta,\Lambda})$.  
	
For $s>0$, the $s$-energy of $\mu$ is
\begin{align*}
	I_s(\mu)
	&=\iint_{\mathbb R^3\times\mathbb R^3}
	|w-z|^{-s}\,d\mu(w)\,d\mu(z)\\
	&=\int_{[0,1]^2}
	\frac{dx\,dy}
	{\bigl((x-y)^2+
		|f_{\Theta,\Lambda}(x)-f_{\Theta,\Lambda}(y)|^2\bigr)^{s/2}}.
\end{align*}
By the energy criterion and Lemma~\ref{lem-lowerbound},
it therefore suffices to prove that
$\mathbb{E}\left(I_s(\mu)\right)<\infty$ for every $2<s<3-2\beta$. 

Tonelli's theorem gives
	\begin{equation*}
		\mathbb{E}\left(I_s(\mu)\right)= \int_{[0,1]^2} \mathbb{E}\left(\frac{1}{\bigl((x-y)^2 + |f_{\Theta,\Lambda}(x)-f_{\Theta,\Lambda}(y)|^2\bigr)^{s/2}}\right) \, dx\,dy.
	\end{equation*}
	
	We claim that there exists $c_0>0$ such that for \( 0 < |x-y| \le \frac{1}{2b^2} \),  
	$$\mathbb{E}\left(\frac{1}{\bigl((x-y)^2 + |f_{\Theta,\Lambda}(x)-f_{\Theta,\Lambda}(y)|^2\bigr)^{s/2}}\right)\le c_0|x-y|^{2-s-2\beta}.$$
	For $|x-y|>1/(2b^2)$, the graph energy kernel is bounded by $(2b^2)^s$. Near the diagonal, $2-s-2\beta>-1$ since $s<3-2\beta$, and
$\int_0^1\int_0^1|x-y|^{2-s-2\beta}\,dx\,dy<\infty$.
Thus the claim implies $\mathbb E\left(I_s(\mu)\right)<\infty$.
	
	Fix \( x,y \in [0,1] \) satisfying \( 0 < |x-y| \le \frac{1}{2b^2}\), and define
	$$X = |f_{\Theta,\Lambda}(x)-f_{\Theta,\Lambda}(y)|^2.$$
	Regarding \( X \) as a random variable depending on \( (\Theta,\Lambda) \), we shall prove below that $X$ has a density $h$ satisfying \eqref{eq-claim}. Assuming this estimate for the moment,
	\begin{align*}
			& \mathbb{E}\left(\frac{1}{\bigl((x-y)^2 + |f_{\Theta,\Lambda}(x)-f_{\Theta,\Lambda}(y)|^2\bigr)^{s/2}}\right) \\
			& \quad= \int_0^\infty \frac{h(t)\,dt}{\bigl((x-y)^2 + t\bigr)^{s/2}} \\
			& \quad= \int_0^\infty \frac{h\bigl(|x-y|^2 \tau\bigr) |x-y|^2 d\tau}{|x-y|^s (1+\tau)^{s/2}} \\
			&\quad \le\|h\|_{L^\infty(0,\infty)} \, |x-y|^{2-s} \int_0^\infty \frac{d\tau}{(1+\tau)^{s/2}} \\
			&\quad = \|h\|_{L^\infty(0,\infty)} \,\frac{2}{s-2}\, |x-y|^{2-s}.
	\end{align*}
	 Hence to prove the claim, we need to show that there exists a constant \( c_1>0 \), independent of \( x,y \), such that
	\begin{equation}\label{eq-claim}
	h(t) \le c_1|x-y|^{-2\beta}.
	\end{equation}
	
	By using the notation in \eqref{eq-phi-n} and \eqref{eq- vector-valued funct.}, $X$ can be written as
	$$X = (a_1(x,y))^2 +(a_2(x,y))^2 := X_1^2 + X_2^2,$$
	where    
	   $$X_1 = -a_1(x,y) = \sum_{n=0}^\infty  \varphi_n \sin\left(\pi\left(b^n x + b^n y +2\theta_n\right)\right)$$ 
	and 
	   $$X_2 = a_2(x,y) = \sum_{n=0}^\infty  \varphi_n \cos\left(\pi\left(b^n x + b^n y +2\lambda_n\right)\right).$$

For every level with $\varphi_n\ne0$, the corresponding sine term has density
\[
h_n(z)=
\begin{cases}
(\pi\sqrt{\varphi_n^2-z^2})^{-1},& |z|<|\varphi_n|,\\
0,& |z|\geq|\varphi_n|.
\end{cases}
\]
When $\varphi_n=0$ its law is a point mass, so no density is asserted at that level. We shall select three nonzero terms and convolve the density of their sum with the probability law of the remaining sum.

Since $0<|x-y|\leq1/(2b^2)$, there is an integer $k\geq2$ such that $\frac{1}{2b^{k+1}} < |x-y| \leq \frac{1}{2b^k}$. 	Then
	\begin{equation*}
		\frac{\pi}{2b^3} < \Bigl|\pi b^{k-2}(y-x)\Bigr|\le \Bigl|\pi b^k(y-x)\Bigr| \le \frac{\pi}{2}.
	\end{equation*}
	So for \( n = k-2,k-1,k \),
	\begin{equation*}
		|\varphi_n|=|2b^{-\beta n}\sin(\pi b^n(x-y))| > 2\sin\Bigl(\frac{\pi}{2b^3}\Bigr) b^{-k\beta}
		\geq 2^{1+\beta}\sin\Bigl(\frac{\pi}{2b^3}\Bigr) |x-y|^\beta.
	\end{equation*}
	
	Let \( \|\cdot\|_p \) denote the \( L^p \)-norm. For \( p=3/2 \),
	\begin{eqnarray*}
		\|h_n\|_{\frac32}
			&=& \left( \int_{-|\varphi_n|}^{|\varphi_n|} \frac{dz}{\bigl(\pi\sqrt{{\varphi_n}^2-z^2}\bigr)^{\frac32}} \right)^{\frac23} \\
			&=& \pi^{-1} \left( \int_{-|\varphi_n|}^{|\varphi_n|} ({\varphi_n}^2-z^2)^{-\frac34}dz \right)^{\frac23} \\
			&=& \pi^{-1} I^{\frac23} |\varphi_n|^{-\frac13},
	\end{eqnarray*}
	where \( I = \int_{-1}^1 (1-u^2)^{-\frac34}du < \infty \) is a universal constant. Thus for \( n=k-2,k-1,k \),
	\[
	\|h_n\|_{\frac32} \le c_2 |x-y|^{-\frac{\beta}{3}}
	\]
	where $c_2$ depends on $b$ and $\beta$, but not on $x,y$. 	By Young's inequality and H\"older's inequality,
	\begin{eqnarray*}
		\bigl(h_{k-2} * h_{k-1} * h_k\bigr)(t)
			&\le& \|h_{k-2}\|_{\frac32} \|h_{k-1} * h_k\|_3 \\
			&\le&  \|h_{k-2}\|_{\frac32} \|h_{k-1}\|_{\frac32}\|h_k\|_{\frac32} \\
			& \le& c_3 |x-y|^{-\beta}
	\end{eqnarray*} 
	where $c_3 :=(c_2)^3$. Let $Y_1$ be the sum of these three terms in $X_1$ and $Z_1$ the sum of the remaining terms. They are independent, and $Y_1$ has density $q=h_{k-2}*h_{k-1}*h_k$. Thus $X_1$ has density $h_{X_1}(t)=\int q(t-z)\,d\mathbb P_{Z_1}(z)$, where $\mathbb P_{Z_1}$ denotes the probability law of $Z_1$. Consequently $\|h_{X_1}\|_\infty\leq\|q\|_\infty\leq c_3|x-y|^{-\beta}$. The cosine terms have the same single-level distributions, so the identical argument gives this bound for $h_{X_2}$.
	
	Finally, we compute the density function of $X= X_1^2 + X_2^2$. For \( i=1,2 \), the density of \( X_i^2 \) is
	\begin{equation*}
		h_{X_i^2}(t) =
		\begin{cases}
			\dfrac{h_{X_i}(\sqrt{t})+h_{X_i}(-\sqrt{t})}{2\sqrt{t}}, & t>0, \\
			0, & t\leq0.
		\end{cases}
	\end{equation*}
	The value assigned to a density at zero is immaterial. Since $X_1$ and $X_2$ depend on independent phase sequences, the density of $X$ satisfies, for $t>0$,
	\begin{eqnarray*}
			h(t) &=& \bigl(h_{X_1^2} * h_{X_2^2}\bigr)(t) \\
			&=& \int_0^t \frac{h_{X_1}(\sqrt{t-r})+h_{X_1}(-\sqrt{t-r})}{2\sqrt{t-r}}
			\cdot \frac{h_{X_2}(\sqrt{r})+h_{X_2}(-\sqrt{r})}{2\sqrt{r}} dr \\
			&\le& ({c_3})^2|x-y|^{-2\beta} \int_0^t \frac{dr}{\sqrt{r(t-r)}} \\
			&=& ({c_3})^2\pi |x-y|^{-2\beta}.
	\end{eqnarray*}
Taking $c_1=\pi (c_3)^2$ proves \eqref{eq-claim}, and hence finite expected energy. 

Choose a sequence $s_j\in(2,3-2\beta)$ increasing to
$3-2\beta$. For each $j$, the energy $I_{s_j}(\mu)$ is finite on an event of probability one. On the intersection of these events, the energy criterion gives 
\[
\dim_H G(f_{\Theta,\Lambda})\geq s_j
\qquad\text{for every }j.
\]
Letting $j\to\infty$ yields $\dim_H G(f_{\Theta,\Lambda})\geq3-2\beta$.  Combining the two cases of $\beta$ with Lemma~\ref{lem-lowerbound}, we obtain
\[
\dim_H G(f_{\Theta,\Lambda})
=\min\{1/\beta,3-2\beta\}
\qquad\text{a.s}.
\]
Finally, since
\[
\dim_H G(f_{\Theta,\Lambda})
\leq\underline{\dim}_B G(f_{\Theta,\Lambda})
\leq\overline{\dim}_B G(f_{\Theta,\Lambda}),
\]
the upper bound in Lemma~\ref{lem-lowerbound} shows that both box dimensions equal the same value almost surely. This completes the proof.
\end{proof}

\bigskip

\noindent {\bf Acknowledgements:} 
The first draft of this paper was completed while the first author was visiting the Department of Mathematics at The Chinese University of Hong Kong in the summer of 2023. He would like to express his sincere gratitude to Prof. De-Jun Feng for his hospitality. The authors are also deeply grateful to Dr. Peng-Fei Zhang for many valuable discussions and suggestions.


\bigskip

\begin{thebibliography}{99}
	
	\bibitem{Baranski2002}
	K.~Bara\'nski,
	On the complexification of the Weierstrass non-differentiable function.
	\textit{Annales Academiae Scientiarum Fennicae Mathematica}
	\textbf{27} (2002), no.~2, 325--339.
	
	\bibitem{Baranski2012}
	K.~Bara\'nski,
	On the dimension of graphs of Weierstrass-type functions with rapidly
	growing frequencies.
	\textit{Nonlinearity} \textbf{25} (2012), no.~1, 193--209.
	
	\bibitem{Baranski2015}
	K.~Bara\'nski,
	Dimension of the graphs of the Weierstrass-type functions.
	In: \textit{Fractal Geometry and Stochastics V},
	Progress in Probability, vol.~70,
	Birkh\"auser/Springer, Cham, 2015, pp.~77--91.
	
	\bibitem{Baranski2014}
	K.~Bara\'nski, B.~B\'ar\'any and J.~Romanowska,
	On the dimension of the graph of the classical Weierstrass function.
	\textit{Advances in Mathematics} \textbf{265} (2014), 32--59.
	
	\bibitem{Falconer2014}
	K.~J.~Falconer,
	\textit{Fractal Geometry: Mathematical Foundations and Applications}.
	Third edition, John Wiley \& Sons, Chichester, 2014.
	
	\bibitem{Hardy1916}
	G.~H.~Hardy,
	Weierstrass's non-differentiable function.
	\textit{Transactions of the American Mathematical Society}
	\textbf{17} (1916), no.~3, 301--325.
	
	\bibitem{Heurteaux2003}
	Y.~Heurteaux,
	Weierstrass functions with random phases.
	\textit{Transactions of the American Mathematical Society}
	\textbf{355} (2003), no.~8, 3065--3077.
	
	\bibitem{Hunt1998}
	B.~R.~Hunt,
	The Hausdorff dimension of graphs of Weierstrass functions.
	\textit{Proceedings of the American Mathematical Society}
	\textbf{126} (1998), no.~3, 791--800.
	
	\bibitem{Kahane1963}
	J.-P.~Kahane, M.~Weiss and G.~Weiss,
	On lacunary power series.
	\textit{Arkiv f\"or Matematik} \textbf{5} (1963), 1--26.
	
	\bibitem{Kaplan1984}
	J.~L.~Kaplan, J.~Mallet-Paret and J.~A.~Yorke,
	The Lyapunov dimension of a nowhere differentiable attracting torus.
	\textit{Ergodic Theory and Dynamical Systems}
	\textbf{4} (1984), no.~2, 261--281.
	
	\bibitem{LaiLauZhang2026}
	C.-K.~Lai, K.-S.~Lau and P.-F.~Zhang,
	Hausdorff dimension of images and graphs of some random complex series.
	Preprint, arXiv:2603.05986 [math.CA], 2026.
	
	\bibitem{Ledrappier1992}
	F.~Ledrappier,
	On the dimension of some graphs.
	In: \textit{Symbolic Dynamics and Its Applications},
	Contemporary Mathematics, vol.~135,
	American Mathematical Society, Providence, RI, 1992, pp.~285--293.
	
	\bibitem{Mandelbrot1977}
	B.~B.~Mandelbrot,
	\textit{Fractals: Form, Chance, and Dimension}.
	W.~H.~Freeman, San Francisco, 1977.
	
	\bibitem{Mattila2015}
	P.~Mattila,
	\textit{Fourier Analysis and Hausdorff Dimension}.
	Cambridge Studies in Advanced Mathematics, vol.~150,
	Cambridge University Press, Cambridge, 2015.
	
	\bibitem{Mauldin1986}
	R.~D.~Mauldin and S.~C.~Williams,
	On the Hausdorff dimension of some graphs.
	\textit{Transactions of the American Mathematical Society}
	\textbf{298} (1986), no.~2, 793--803.
	
	\bibitem{Przytycki1989}
	F.~Przytycki and M.~Urba\'nski,
	On the Hausdorff dimension of some fractal sets.
	\textit{Studia Mathematica} \textbf{93} (1989), no.~2, 155--186.
	
	\bibitem{Ren2023}
	H.~Ren,
	Box dimension of the graphs of the generalized Weierstrass-type functions.
	\textit{Discrete and Continuous Dynamical Systems}
	\textbf{43} (2023), no.~10, 3830--3838.
	
	\bibitem{RenShen2021}
	H.~Ren and W.~Shen,
	A dichotomy for the Weierstrass-type functions.
	\textit{Inventiones Mathematicae}
	\textbf{226} (2021), 1057--1100.
	
	\bibitem{RenShen2024}
	H.~Ren and W.~Shen,
	The high-dimensional Weierstrass functions.
	Preprint, arXiv:2404.06778 [math.CA], 2024.
	
	\bibitem{Rom2014}
	J.~Romanowska,
	Measure and Hausdorff dimension of randomized Weierstrass-type functions.
	\textit{Nonlinearity} \textbf{27} (2014), no.~4, 787--801.
	
	\bibitem{Salem1945}
	R.~Salem and A.~Zygmund,
	Lacunary power series and Peano curves.
	\textit{Duke Mathematical Journal}
	\textbf{12} (1945), no.~4, 569--578.
	
	\bibitem{Shen2018}
	W.~Shen,
	Hausdorff dimension of the graphs of the classical Weierstrass functions.
	\textit{Mathematische Zeitschrift}
	\textbf{289} (2018), 223--266.
	
	\bibitem{Solomyak1995}
	B.~Solomyak,
	On the random series $\sum \pm\lambda^n$ (an Erd\H{o}s problem).
	\textit{Annals of Mathematics}
	\textbf{142} (1995), no.~3, 611--625.
	
	\bibitem{Weierstrass1872}
	K.~Weierstrass,
	\"Uber continuirliche Functionen eines reellen Arguments, die f\"ur
	keinen Werth des letzteren einen bestimmten Differentialquotienten besitzen.
	In: \textit{Mathematische Werke}, vol.~2,
	Mayer \& M\"uller, Berlin, 1895, pp.~71--74.
	
\end{thebibliography}
\end{document}